\documentclass[12pt]{amsart}
\usepackage{amscd,amssymb}
\usepackage[graph,frame,poly,arc]{xy}  
\usepackage{tikz}
\usepackage[plainpages,backref,urlcolor=blue]{hyperref}

\theoremstyle{plain}
\newtheorem{thm}[subsection]{Theorem}

\newtheorem{prop}[subsection]{Proposition}
\newtheorem{cor}[subsection]{Corollary}

\theoremstyle{definition}
\newtheorem{rk}[subsection]{Remark}

\newtheorem{ex}[subsection]{Example}
\newtheorem{conj}[subsection]{Conjecture}
\newtheorem{question}[subsection]{Question}

\numberwithin{equation}{section}
\newcommand{\OO}{{\mathcal O}}

\newcommand{\C}{\mathbb{C}}
\newcommand{\PP}{\mathbb{P}}

\begin{document}
\date{March 21, 2021}

\title [On the Bounded Negativity Conjecture ]
{On the Bounded Negativity Conjecture and singular plane curves}

\author[Alexandru Dimca]{Alexandru Dimca$^{1}$}
\address{Universit\'e C\^ ote d'Azur, CNRS, LJAD, France and Simion Stoilow Institute of Mathematics,
P.O. Box 1-764, RO-014700 Bucharest, Romania}
\email{dimca@unice.fr}

\author[Brian Harbourne]{Brian Harbourne$^2$}
\address{Math Department\\
University of Nebraska--Lincoln\\
Lincoln, NE 68588 USA}
\email{brianharbourne1@unl.edu}

\author[Gabriel Sticlaru]{Gabriel Sticlaru}
\address{Faculty of Mathematics and Informatics,
Ovidius University
Bd. Mamaia 124, 900527 Constanta,
Romania}
\email{gabriel.sticlaru@gmail.com }

\thanks{\vskip0\baselineskip
\vskip-\baselineskip
\noindent $^1$This work has been partially supported by the Romanian Ministry of Research and Innovation, CNCS - UEFISCDI, grant PN-III-P4-ID-PCE-2020-0029, within PNCDI III.
\vskip0\baselineskip
\noindent $^2$The second author benefitted from the support of Simons Foundation grant \#524858.}



\begin{abstract} 
There are no known failures of Bounded Negativity in characteristic 0.
In the light of recent work showing the Bounded Negativity Conjecture fails in positive characteristics
for rational surfaces, we propose new characteristic free conjectures as a replacement.
We also develop bounds on numerical characteristics of curves  
constraining their negativity. For example, we show that the $H$-constant of a rational curve
$C$ with at most $9$ singular points satisfies $H(C)>-2$ regardless of the characteristic. 
\end{abstract}
 
\maketitle

%\tableofcontents

\section{Introduction} 

Except where explicitly noted, we work over an arbitrary algebraically closed field.
Also, when we speak of the singular points of a plane curve,
we mean including all infinitely near singular points, if any, unless we specify otherwise.

The following folklore conjecture, which has been traced back at least to F. Enriques (see \cite{BHKMS}), is still open.

\begin{conj}
\label{c0}
 (Bounded Negativity Conjecture). Let $X$ be any smooth complex projective
surface. Then there is a bound $B_X$ such that for every reduced curve $C \subset  X$ we have
$C^2 \geq B_X$.
\end{conj}

An equivalent formulation (see \cite[Proposition 5.1]{BHKMS}) is as follows.
 
\begin{conj}
\label{c0irr}
(Bounded Negativity Conjecture for irreducible curves). Let $X$ be any smooth complex projective
surface. Then there is a bound $b_X$ such that for every reduced irreducible curve $C \subset X$ we
have $C^2 \geq b_X$.
\end{conj}

It has been known for a long time that counterexamples occur in positive characteristics 
(see \cite[Exercise V.1.10]{Hr}), and recently counterexamples have been found even for 
rational surfaces \cite{CV} and hence every smooth projective surface in positive characteristic
is birational to a surface for which Bounded Negativity fails. This raises the question
of whether a characteristic free formulation of Bounded Negativity can be given.

As a step in this direction, we will consider the case of $X=\PP^2$.
The failure of Bounded Negativity for rational surfaces implies the existence
of plane curves with singular points of high multiplicity compared to the degree,
at a fixed set of points of the plane. What has never been observed in any characteristic 
but which has not been
shown to be impossible, are failures of what can be regarded Bounded Negativity in which one does not bound 
the number of singular points of multiplicity greater than 2 but only their multiplicities and the genus of
the curves. 

It is easy to construct curves of low genus with large numbers of double points (see Example \ref{xxx}).
But very few examples of reduced curves are known, in any characteristic, with bounded 
genus if we bound the number of double points and bound the multiplicities of the remaining singular points.
This motivates the following conjectures.
They are consistent with what we currently know.
We pose them both as a way of encapsulating our current state of knowledge,
and in the hope that studying them will enhance our understanding of bounded negativity.

To state them, we will call a positive integer sequence $(d,m_1,\ldots,m_r)$ with $r>0$
a {\it multiplicity sequence} if there is a reduced plane curve $C$ of degree $d$
having exactly $r$ singular points $p_1,\ldots,p_r$ (possibly infinitely near) and 
such that $m_i$ is the multiplicity of $p_i$. 
We say $C$ has {\it multiplicity bound} $m$ if $m_i\leq m$ for all $i$.

\begin{conj}
\label{c1.2}
Let $K$ be a fixed algebraically closed ground field.
There are only finitely many
multiplicity sequences arising for reduced plane curves over $K$ with multiplicity bound $m=3$,
having no points of multiplicity 2 and
whose components are all rational.
\end{conj}

More generally, we call $(g_1,\ldots,g_s)$ a {\it genus sequence}
if there is a reduced plane curve with components $C_1,\ldots,C_s$
where $g_j$ is the genus of the normalization $C_j'$ of $C_j$.
We say $C$ has genus bound $g$ if $g_j\leq g$ for all $j$.

\begin{conj}
\label{c1.3a}
Let $K$ be a fixed algebraically closed ground field.
For every choice of integers $n\geq0$, $m>0$ and $g\geq0$, there are only finitely many
multiplicity sequences arising for irreducible plane curves over $K$ with multiplicity bound $m$,
genus bound $g$ and having at most $n$ singular points of multiplicity less than $m$.
\end{conj}

Here is a stronger form of this conjecture:

\begin{conj}
\label{c1.3}
Let $K$ be a fixed algebraically closed ground field.
For every choice of integers $n\geq0$, $m>0$ and $g\geq0$, there are only finitely many
multiplicity sequences arising for plane curves over $K$ with multiplicity bound $m$,
genus bound $g$ and having at most $n$ points of multiplicity 2.
\end{conj}

As a more accessible target for a counterexample, we also propose:

\begin{conj}
\label{c1.4}
There is no reduced plane curve in any characteristic of degree more than 9 whose components are all rational
and whose singularities all have multiplicity 3.
\end{conj}

% For example, only four arrangements of lines seem to be known which, when regarded as reducible plane curves,
% have only points of multiplicity 3 (here we have a bound on the genus of the components, namely 0, 
% and on the number of double points, again 0). These are: 
% three concurrent lines (it has multiplicity sequence $(3,3)$),
% the seven lines of the Fano plane (in characteristic 2; it has multiplicity sequence $(3,3_7)$),
% the 9 lines as given by $(x^3-y^3)(x^3-z^3)(y^3-z^3)=0$ in characteristics other than $3$ 
% and the 9 lines $(x+y+z)(x+y-z)(x-y+z)(x-y-z)(x+z)(x-z)(y+z)(y-z)z$ over ${\bf Z}/3{\bf Z}[x,y,z]$
% which do not contain the point $(0,0,1)$ (these both have multiplicity sequence $(9,3_{12})$).
% Only a few more examples are known if we allow rational curves of higher degree.

These conjectures generalize a long standing problem from combinatorics.
Only two complex line arrangements are known whose multiple points all have
multiplicity 3, namely 3 concurrent lines (which has multiplicity sequence $(3,3)$)
and the curve defined by $(x^3-y^3)(x^3-z^3)(y^3-z^3)=0$
(which has multiplicity sequence $(9,3_{12})$, which we use
as shorthand for the sequence with 12 consecutive 3's).
In positive characteristics, we get two more examples.
The union of the 7 lines of the Fano plane gives
an example in characteristic 2 (it has multiplicity sequence $(7,3_7)$). 
For a field $K$ of characteristic 3, one can take the
9 $k$-lines of $\PP^2_k$ that do not contain the point $(0,0,1)$,
where $k\subset K$ is the subfield of order 3; it has multiplicity sequence $(9,3_{12})$.
% These lines are z*(x+z)*(-x+z)*(y+z)*(-y+z)*(x+y+z)*(x-y+z)*(-x+y+z)*(-x-y+z)
It seems not to be known if there are any others in any characteristic.

By allowing curves of higher degrees we obtain a few more
multiplicity sequences with genus bound $g=0$, multiplicity bound $m=3$ and 
having $n=0$ points of multiplicity 2, namely
$(4,3)$, $(5,3_3)$, $(6,3_4)$, $(8,3_7)$, $(8,3_8)$ and $(9,3_{10})$.
We discuss these examples in more detail in the last section of this paper.
We are unaware of any others. Of the examples we are aware of, only $(4,3)$
and $(8,3_7)$ come from irreducible curves.

\subsection{$H$-constants and connections to Bounded Negativity}
Before discussing our results, we first describe how the multiplicity
sequence conjectures relate to the Bounded Negativity Conjecture.
The connection (see Theorem \ref{rkLink}) is via the concept of 
$H$-constants, which was introduced to explore the Bounded Negativity Conjectures
(see for example \cite{BDHLPS, P, Szp1}). Given a reduced singular curve $C\subset \PP^2$, let
$S = \{p_1, ...,  p_r\}$ be the set of singular points of $C$ (including infinitely near singular points;
see \cite{PR}), 
let $m_i$ be the multiplicity of $p_i$ and 
let $d$ be the degree of $C$. We call 
$p_i$ an {\it ordinary singular point} if it is an actual point of $\PP^2$
and has $m_i$ smooth branches, with distinct tangents. 
(Thus no point infinitely near an ordinary singular point can be singular.)
We now define
 
\begin{equation}
\label{eq1}
H(C)=\frac{d^2-\sum_{i=1}^rm_i^2}{r}=\frac{(C')^2}{r},
\end{equation}
where $C'$ is the proper transform of $C$ under the blowing up of the points $p_i$.
Note that $C'$ is a disjoint union of smooth curves, $C_1\ldots,C_s$.

It is easy to see that if $\inf H(C)>-\infty$, where the infimum is taken over either all reduced 
or over all irreducible, reduced
singular curves $C$, then the Bounded Negativity Conjecture is true.
Since by \cite{CV} the Bounded Negativity Conjecture is false
in every positive characteristic, we have $\inf H(C)=-\infty$ in every positive characteristic.
Indeed, over any algebraically closed field $K$
of positive characteristic $p>0$, it is easy to give curves $C$ whose singularities are all ordinary
but where $H(C)$ is arbitrarily negative (see Example \ref{yyy}).
By \cite{CV}, examples in every positive characteristic can even be given with $H(C)$ arbitrarily negative where
$C$ is reduced and irreducible. 

\begin{rk}
Some authors define $H(C)$ by restricting the sum $\sum_{i=1}^rm_i^2$ to multiplicities of actual points of $\PP^2$,
hence ignoring infinitely near singular points.
Since what is of concern with $H$-constants is how negative they can be,
ignoring some of the singular points of a curve $C$ would be of interest if it could make
$H(C)$ more negative, but currently we are not aware of any such examples.
Since each term in the sum $\sum_{i=1}^rm_i^2$ is at least 4, one cannot reduce $H(C)$
by ignoring singularities unless $H(C)<-4$.
Because $H$-constants can be arbitrarily negative in positive characteristics,
this might be a good venue to look for such examples.
No complex curves $C$ are yet known, however, with $H(C)\leq -4$.
Likewise, versions $\sigma^0_k$ of the quantities $\sigma_k$ defined below could
be defined ignoring infinitely near points, and we would for $k\leq 4$ have 
$\sigma^0_k(C)\geq\sigma_k(C)$. For $k>4$ it is a priori possible to
have $\sigma^0_k(C)<\sigma_k(C)$; it would be interesting to know if such an example
exists. 
\end{rk}

The following questions occur in \cite{H}; see Question 2.21 and Question 2.25.

\begin{question}
\label{q1}
Does there exist an irreducible reduced singular curve $C \subset \PP^2$ such that $$H(C)\leq -2 \  \ ?$$
\end{question}

Question \ref{q1red} is open in characteristic 0, where the most negative example known 
comes from the curves $C_d$ in Example \ref{xxx}, where we have $H(C_d)>-2$ for all $d$
but $\lim_{d\to\infty}H(C_d)=-2$. The curves $C_d$ are rational, of degree $d$ and have only double point
singularities. By Example \ref{zzz},
reduced, irreducible rational curves $C$ of degree $d$ with only triple points would have 
would have $H(C)$ approach $-3$ from above as $d$ increases.
So a possible approach to Question \ref{q1} could involve finding 
rational curves of high degree with only triple points, if such exist.

Reducible complex curves are known with $H$-constant close to $-4$, but no such curve $C$ is 
known with $H(C)\leq -4$ (see \cite{BHRS,R,RU}). 
This motivates the next question.

\begin{question}
\label{q1red}
Does there exist a reduced singular complex curve $C \subset \PP^2$ such that $$H(C)\leq -4 \  \  ?$$
\end{question}

We now state our first result. It establishes a connection between our multiplicity sequence conjectures 
and lower bounds on $H$-constants which themselves relate to the Bounded Negativity Conjecture.
Since $H$-constants are not bounded below in positive characteristics,
such a theorem is mainly of interest over the complex numbers even though the 
multiplicity sequence conjectures are of interest in every characteristic.

\begin{thm}\label{rkLink}
Assume over the complex numbers that we have $\inf H(C)>-\mu$ for some integer $\mu>1$.
Then, given any $n$ and $g$, Conjecture \ref{c1.3a} holds for each $m\geq\mu$.
\end{thm}

\begin{rk}
A version of the original Bounded Negativity Conjecture in which one bounds the genus is known to be true.
In particular, given a smooth complex surface $X$ and integer $g\geq0$, there is a bound $b_{X,g}$ such that
$C^2\geq b_{X,g}$ for every reduced curve $C\subset X$, where the geometric genus of each component of 
$C$ is bounded from above by $g$ \cite{Hao}.
\end{rk}

\subsection{Main results}
As Questions \ref{q1} and \ref{q1red} and Theorem \ref{rkLink} show, it is of interest, over the complex numbers,
to study the existence of lower bounds for $H$-constants. 
Our approach to this problem is to identify conditions which 
constrain the negativity of $H$-constants,
and is therefore of interest in all characteristics. We will use the following notation:
\begin{equation}
\label{eq1.1}
\sigma_k(C):=d^2-\sum_{i=1}^rm_i^2+kr.
\end{equation}
With this notation, a negative answer to Question \eqref{q1} is equivalent to the following.
\begin{conj}
\label{c2}
 With the above notation, for any  irreducible reduced singular complex curve $C \subset \PP^2$ one has 
 $$\sigma_2(C)>0.$$
 \end{conj}
 And a negative answer to Question \eqref{q1red} is equivalent to the following.
\begin{conj}
\label{c2red}
 With the above notation, for any  reduced singular complex curve $C \subset \PP^2$  one has 
 $$\sigma_4(C)>0.$$
 \end{conj}

In this note we prove some inequalities for the invariant $\sigma_k(C)$.
No irreducible complex plane curves $C$ are known for which $\sigma_k(C)<0$ when $k\geq2$.
In positive characteristics by \cite{CV} we now have for any $k$ irreducible curves $C$
with $\sigma_k(C)<0$, but it is still of interest to understand better
how these curves arise. 

Our first result is the following.
\begin{thm}
\label{thm1}
With the above notation, assume that $C$ is irreducible. Then one has
$$\sigma_k(C) = 3d-E(\tilde C) -\sum_{i=1}^r(m_i-k),$$
where $\tilde C$ is the normalization of the curve $C$ and $E(\tilde C)$ denotes the corresponding Euler characteristic of $\tilde C$.
Therefore Conjecture \ref{c2} holds exactly when
$$3d-E(\tilde C) >\sum_{i=1}^r(m_i-2).$$
\end{thm}

\begin{cor}
\label{cor-1}
Let $C$ be a reduced plane curve with $r$ singularities, let $\overline{m}$ be the average
of the multiplicities of the singular points and let $m$ be the maximum of the multiplicities.
\begin{enumerate}
\item[(a)] Then $\sigma_k(C)>r(k-\overline{m})$.
\item[(b)] Conjecture \ref{c2} holds when $m=2$.
\item[(c)] Conjecture \ref{c2red} holds when $\overline{m}\leq 4$ and hence when $m\leq4$.
\item[(d)] If $C$ is also irreducible, then Conjecture \ref{c2} holds when
\begin{equation}
\label{eq10}
3d-2 >\sum_{i=1}^r(m_i-2).
\end{equation}
This inequality holds for every irreducible curve of degree $d \leq 20$.
\end{enumerate}
\end{cor}
The fact that Conjecture \ref{c2} holds for irreducible curves of degree $d \leq 20$ was known, for instance it is stated towards the end of section 2  in \cite{BDHLPS}. See Example \ref{ex1} below for curves not satisfying the inequality \eqref{eq10}.

\begin{cor}
\label{cor1}
Assume $C$ is reduced and irreducible.
Conjecture \ref{c2} holds when all the singularities of the curve $C$ 
have multiplicity at most $3$ and all the singularities of multiplicity $3$ are situated on a curve 
of degree 8.  In particular, this holds when the number of singularities of $C$ of multiplicity $3$ is strictly smaller than $45$. 

\end{cor}

Our next result gives bounds on $\sigma_k$ for certain values of $k$.
The proof uses methods different from those used for our results above,
methods we believe to be of independent interest.

\begin{prop}
\label{prop0}
For a reduced singular complex plane curve $C$ whose singular point of maximum multiplicity has multiplicity $m$, one has
$$\sigma_{2m-1}(C)\geq 2d-1.$$
If $C$ is irreducible, then
$$\sigma_{m}(C)\geq 3d-2.$$
\end{prop}

Finally we discuss some consequences of Theorem \ref{thm1} above.
Corollary \ref{cor2} is immediate given the obvious fact that 
$E(\tilde C)=E(C)+\sum_i(m_i-1)$ when the singularities are all ordinary.
And Corollary \ref{cor0} is immediate given that $E(\tilde C)=2$ for a rational curve.

\begin{cor}
\label{cor2}
With the above notation, assume  that all the singularities of the irreducible curve $C$ are ordinary. Then 
 $$\sigma_2(C) = 3d-E(C) -\sum_{i=1}^r(2m_i-3).$$
Hence, in this situation, Conjecture \ref{c2} is equivalent to the inequality
$$3d -\sum_{i=1}^r(2m_i-3)>E(C).$$
\end{cor}

\begin{cor}
\label{cor0}
If $C$ is a singular rational curve with $r$ singular points of multiplicities $m_1,\ldots,m_r$, then 
Conjecture \ref{c2} is equivalent to the inequality
$$3d-2> \sum_{i=1}^r(m_i-2)=n_3+2n_4+3n_5+ \ldots+(d-3)n_{d-1},$$
where $n_j$ denotes the number of singularities of $C$ of multiplicity $j$.
\end{cor}

In spite of the large number of known results on the possible configurations of 
singularities of an irreducible  plane curve, see for instance \cite{GLS,GS}, 
the inequality in Corollary \ref{cor0} does not seem to be available, although 
no example is known over the complex numbers of a plane rational curve for which the inequality does not hold.

The next result gives some idea of where to look for curves for which 
the inequality in Corollary \ref{cor0} does not hold.

\begin{thm}
\label{thmNew}
If $C$ is a singular rational curve of degree $d$ with $r$ singular points, then 
the inequality in Corollary \ref{cor0} holds if
\begin{equation}
\label{eqNew}
(36-4r)d^2+d(r-1)48+r(8r-32)+16>0.
\end{equation}
In particular, the inequality in Corollary \ref{cor0}
holds for all $d>0$ if $r\leq 9$.
However, Inequality \eqref{eqNew} fails for all $d\gg0$ if $r>9$.
For example, it holds for no $d\geq 110$ if $r=10$.
\end{thm}

\begin{rk}
\label{rk0}
Note that the inequality in Corollary \ref{cor0} puts no constraints on $n_2$.
Indeed, it is known for each $d\geq 1$ that there 
are degree $d$ rational curves whose singularities are all nodes. Thus the number $n_2$ of nodes is
the maximum allowed by the genus formula, namely
$$n_2= \frac{(d-1)(d-2)}{2}.$$
These curves are discussed in more detail in Example \ref{xxx} (also see \cite{Oka}).
\end{rk}

\begin{rk}
\label{rk21} Assume that $C$ is a rational plane curve, having only nodes and ordinary triple points as singularities. Then 
the condition $n_3<3d-2$ in Corollary \ref{cor0} is equivalent by the genus formula to 
$$n_2> \frac{d^2-21d+14}{2}.$$
Hence one way to construct counterexamples to Conjecture \ref{c2} 
might be to construct a rational plane curve, having only nodes and ordinary triple points 
as singularities, of degree $d>20$ with as few nodes as possible.
We have no idea whether such counterexamples exist, but lower degree examples suggest the following
variation of Conjecture \ref{c1.4}.
\end{rk}

\begin{conj}
\label{c20}
Let $C$ be a reduced plane curve having only nodes and triple points as singularities.
If $C$ has degree $d>9$, or if $C$ is irreducible and has degree $d>8$, then the number $n_2$ of nodes
is strictly positive.
\end{conj}
This conjecture for irreducible $C$ is obviously true when $d \equiv 0$ (mod $3$), and it holds
for $d \geq 21$  if Conjecture \ref{c2} holds
in this range. In view of the existence of rational nodal curves, as recalled in Remark \ref{rk0}, 
Conjecture \ref{c20} looks quite surprising. 
In contrast, for irreducible complex $C$ we have, as far as we know, only 
quartics (such as $C:(x^2+y^2)^2+y^3-3yx^2=0$,
a rational curve with one ordinary triple point and no other singularity; see Figure \ref{rose})
and octics (such as the complex rational octic curve with $n_3=7$ triple points
shown in Figure \ref{irredOctic}). These two examples occur (with different equations)
in every characteristic. Thus the condition $d>8$ in Conjecture \ref{c20} is necessary.
Note that in degrees $d>8$, $d \not \equiv 0$ (mod $3$), the non-existence of plane curves with 
$$n_3=\frac{(d-1)(d-2)}{6}$$
triple {\it generic} points follows from \cite{Dumi}, since the expected dimension in this case is negative.
In low degrees $d\leq 7$, we have the following easy result.

\begin{prop}
\label{prop1}
For a rational plane curve, having only nodes and  
triple points as singularities, of degree  $d \in \{5,6,7\}$, the number $n_2$ of nodes
is strictly positive. 
\end{prop}

\section{The proofs of Theorem \ref{rkLink}, Theorem \ref{thm1},  Corollary \ref{cor-1} and Corollary \ref{cor1}} 

We begin by rewriting $H$-constants using adjunction (i.e., using the genus formula).

\begin{prop}\label{genusFormula}
Let $C$ be a reduced plane curve of degree $d$ with $r$ singular points $p_i$ of multiplicities $m_i$ and $s$ components $B_j$.
Let $g_j$, $j=1,\ldots,s$ be the genus of the normalization $A_j$ of $B_j$.
Then we have 
$$H(C)=\frac{3d+\sum_j(2g_j-2)-\sum_{i=1}^rm_i}{r}\geq \frac{3d-2s-\sum_{i=1}^rm_i}{r}\geq \frac{d-\sum_{i=1}^rm_i}{r}>-\overline{m},$$
where $\overline{m}$ is the average of the multiplicities of the singular points. 
Moreover, it is clear that the first inequality is an equality if and only if all of the components of $C$ are rational, and the second
inequality is an equality if and only if all of the components are lines.
\end{prop}

\begin{proof}
Let $X_1\to\PP^2$ be the morphism obtained by blowing up the points of $\PP^2$ where $C$ is singular,
and let $X_2\to X_1$ be the morphism obtained by blowing up the points 
at which the proper transform $C_1\subset X_1$ of $C$ is singular (these will all be points infinitely near
points already blown up). Recursively, let
$X_{i+1}\to X_i$ be the blow up of the singular points of the proper transform $C_i$ of $C$ to $X_i$,
and let $C_{i+1}$ be the proper transform of $C$ to $X_{i+1}$.
For some $t$ we will have blown up all of the singular points. Let $X=X_t$. Then $C'=C_t\subset X$ is smooth. 

We have $(C')^2=d^2-\sum_i m_i^2$ and $C'\cdot K_{X_t}=-3d+\sum_i m_i$.
Let $A_1+\cdots+A_s$ be the components of $C'$. These are smooth by construction; 
let $g(A_j)$ be the genus of the curve $A_j$.
By adjunction we have $A_j^2+A_j\cdot K_{X_t}=2g_j-2$, and since the components are (again by construction) disjoint,
we have $(C')^2+C'\cdot K_{X_t}=\sum_jA_j^2+\sum_j A_j\cdot K_{X_t}=\sum_j 2g_j-2$.
Thus 
\begin{equation}
\label{genusEquation}
d^2-\sum_i m_i^2=\sum_j (2g_j-2)+3d-\sum_i m_i,
\end{equation}
or, equivalently,
\begin{equation}
\label{genusEquation2}
\binom{d-1}{2}-\sum_i\binom{m_i}{2}-(s-1)=\sum_jg_j,
\end{equation}
so
$$H(C)=\frac{d^2-\sum_i m_i^2}{r}=\frac{\sum_j (2g_j-2)+3d-\sum_i m_i}{r}.$$
We have the obvious facts that $g_j\geq0$ and $d\geq s$, so 
$$\frac{\sum_j (2g_j-2)+3d-\sum_i m_i}{r}\geq \frac{3d-2s-\sum_i m_i}{r}\geq  \frac{d-\sum_i m_i}{r}>\overline{m}.$$
\end{proof}

\begin{proof}[The proof of Theorem \ref{rkLink}]
Suppose over the complex numbers we have $\inf H(C)>-\mu$ for some integer $\mu>1$.
Given any $n$ and $g$, to show Conjecture \ref{c1.3a} must hold for each $m\geq\mu$,
it suffices to show that failure of Conjecture \ref{c1.3a} for given $n$, $g$ and $m$ implies
$\inf H(C)\leq -m$.
Indeed, failure of Conjecture \ref{c1.3a} for given $n$, $g$ and $m$ implies
we have an infinite sequence of reduced, irreducible plane curves $C_i$ with
$r_i$ points of multiplicity $m$ with $r_1<r_2<\cdots$.
Let $d_i$ be the degree of $C_i$. Since the number $r_i$ of points of multiplicity $m$ is increasing,
for each $i$ there is a $j$ such that $d_j>d_i$, hence by picking a subsequence we may assume
that $d_i$ is increasing too.
Let $\gamma_i\leq g$ be the genus of the normalization of $C_i$ and assume 
the other singular points of $C_i$ have multiplicities $m_{i1}\leq m_{i2}\leq \cdots\leq m_{in_i}<m$ with $n_i\leq n$.
Then we have
$$\binom{d_i-1}{2}-\sum_{j=1}^{n_i}\binom{m_{ij}}{2} -r_i\binom{m}{2}=\gamma_i,$$
hence $r_i=(\binom{d_i-1}{2}-(\sum_{j=1}^{n_i}\binom{m_{ij}}{2}+\gamma_i))/\binom{m}{2} 
\geq ((d_i-1)(d_i-2)-(nm^2+2g))/(m(m-1))$.
Thus $r_i/d_i$ has limit 0, so we have 
$$\inf_i H(C_i)=\inf_i ((d_i-\sum m_i-r_im)/(n_i+r_i))\leq \inf_i ((d_i/r_i)-r_im/(n+r_i))=-m.$$
\end{proof}

Theorem \ref{thm1} is an immediate corollary of Proposition \ref{genusFormula},
so we proceed to Corollary \ref{cor-1}.

\begin{proof}[The proof of Corollary \ref{cor-1}]
(a) By Proposition \ref{genusFormula},
we have $H(C)>-\overline{m}$, hence $\sigma_k(C)=rH(C)+rk>r(k-\overline{m})$.
Now (b) and (c) are immediate.

(d) The first claim follows immediately from  Theorem \ref{thm1}, since $E(\tilde C) \leq 2$.
For the second claim, 
we have by Equation \ref{genusEquation2} (the genus formula with $s=1$) that
\begin{equation}
\label{eq23}
\binom{d-1}{2}\geq  \sum_{i=1}^r \binom{m_i}{2}.
\end{equation} 
But
$$\binom{m}{2} \geq 3(m-2)$$
for integers $m \geq 2$, so it follows that
$$\frac{(d-1)(d-2)}{6} \geq  \sum_{i=1}^r (m_i -2)$$
To end the proof of Corollary \ref{cor-1}, we observe that
$$3d-2 > \frac{(d-1)(d-2)}{6} $$
for any degree $d \leq 20$.
\end{proof}

%\newpage

\begin{ex}
\label{ex1} We give three examples of irreducible curves $C$ for which 
the inequality in \eqref{eq10} does not hold. 

(a) In positive characteristics, \cite{CV} constructs rational curves
for which $H(C)<-2$, and hence $\sigma_2(C)<0$, and therefore
the inequality in \eqref{eq10} does not hold. 

(b) We give two more examples, (i) and (ii), 
for which Inequality \eqref{eq10} does not hold. 

(i) Here we construct explicit complex irreducible curves $C$ for which 
the inequality in \eqref{eq10} does not hold. 
Let $3 \leq p <q $ be two relatively prime integers and consider the
irreducible curve
$$C_{p,q}: (x^p+y^p)^q-(y^q+z^q)^p=0.$$
As usual let $p_1,\ldots,p_r$ be the singular points of $C_{p,q}$ (including infinitely near singular points), 
but assume $p_1,\ldots,p_s$, for some $s\leq r$, are actual points of $\PP^2$.
One has $d=s=pq$ and $m_i=p$ for $1\leq i\leq s$, hence
$$3d-2 -\sum_{i=1}^r(m_i-2)\leq 3d-2 -\sum_{i=1}^s(m_i-2)= 3pq-2-pq(p-2)= pq(5-p)-2 <0$$
if $p\geq 5$. 
Note all of the singularities of $C$ are irreducible; this means $C$ is irreducible
(otherwise $C$ would have a reducible singularity where components meet).
This also means $E(\tilde C)=E(C)$, and hence
$$E(\tilde C)=2-(d-1)(d-2)+\sum_{i=1}^s\mu_i=pq(4-p-q),$$
where $\mu_i=(p-1)(q-1)$ is the Milnor number of the singularity $p_i$ (so the genus of 
$\tilde C$ is $pq((p+q-4)+2)/2$).

(ii) Now we give a characteristic free construction of irreducible curves $C$ 
where Inequality \eqref{eq10} does not hold. 
In this example, the singular points all are ordinary, hence there are no infinitely near singular points.
The genus of the curves in this example is large, which explains how 
Inequality \eqref{eq10} can fail while having $H(C)>-2$.
Pick 16 points on a smooth quartic plane curve $Q$. Pick a general element of the linear system
of curves of degree $d=5m$ vanishing to order $m$ at each of the 16 points.
For $m$ sufficiently large,
we will show that a general curve $C$ in this linear system will be irreducible and have ordinary singular points.
However $H(C)=9m^2/16$ while Inequality \eqref{eq10} states $15m-2>16(m-2)$, which fails for $m\geq 30$.
For $m\gg0$, the genus of $C$ is $g=(9m^2+m+2)/2$.
To justify our claims, let $X\to\PP^2$ be the blow up of the 16 points.
Let $C'$, $Q'$ and $L$ be the proper transforms of $C$, $Q$ and a general line, respectively.
Then $C'$ is linearly equivalent $m(Q'+L)$. But $Q'$ is nef so $Q'+L$ is ample, hence $C'$ is very ample
for $m\gg0$. By Bertini's Theorem (e.g., \cite[Theorem II.8.18]{Hr}), $C'$ is smooth and irreducible
and meets each of the 15 exceptional curves smoothly. Thus $C$ is irreducible and has only ordinary singular points.
\end{ex}

\begin{proof}[The proof of Corollary \ref{cor1}]
In view of Corollary \ref{cor-1}, we may assume that $d>20$.
As in the proof of Proposition \ref{genusFormula}, let $\pi:X\to\PP^2$ be the blow up 
of the singular points (including all infinitely near singular points).
Let $p_1,\ldots,p_r$ be the singular points and let $E_i$ be the exceptional locus
corresponding to $p_i$. (I.e., in the construction of $X\to\PP^2$ given in the proof of 
Proposition \ref{genusFormula}, if $p_i$ is an actual point of $X_j$, then
$E_i$ is the scheme theoretic inverse image of $p_i$ under $X\to X_j$.)
Let $L$ be the inverse image under $\pi$ of a general line in $\PP^2$.
Then the proper transform $\tilde C$ of $C$ on $X$ is linearly equivalent to
$dL-\sum_i m_iE_i$.

Let $C'$ be a curve of degree 8 containing the singularities of $C$ of multiplicity $3$
(i.e., $C'$ is the image of an effective divisor $C''$ on $X$ linearly equivalent to
$8L-\sum_{i_j}E_{i_j}$ where the sum is over those $i$ such that $m_i=3$.)

Since $C$ and hence $\tilde C$ is irreducible of degree greater than the degree of $C'$,
$\tilde C$ and $C''$ have no common irreducible component.
Hence $0\leq (\tilde C)\cdot C'' = 8d-3n_3$, so 
$$8d \geq   3n_3, $$
where  $n_3$ denotes the number of singular points of multiplicity $3$ on $C$.
But $d>20$, so
$$n_3 \leq \frac {8d}{3} < 3d-2$$
and the first claim follows from Corollary \ref{cor-1}(d).
For the second claim, note that the space of homogeneous polynomials in 3 variables of degree 8 has dimension 
$$\binom{10}{2}=45.$$
Therefore, there is such a nonzero homogeneous polynomial of degree 8 
vanishing at all of the triple points of $C$, since $n_3<45$. 
(More precisely, by Riemann-Roch for rational surfaces we have 
$h^0(X,\OO_X(C'')) \geq \frac{1}{2}((C'')^2-K_X\cdot C'')+1
= 45-n_3>0$.)
\end{proof}

\section{The proof of Proposition \ref{prop0},
of Theorem \ref{thmNew} and of Proposition \ref{prop1}} 

\begin{proof}[Proof of Proposition \ref{prop0}] 
Let $(C,q)$ be an isolated plane curve singularity. Then the usual definition of the 'double point number' $\delta(C,q)$ using the normalization morphism can be used, and the formula
$\delta(C,q)$  in terms of the multiplicities $m_j$ of the infinitely near points $q_j$, namely
\begin{equation}
\label{eq3}
\delta(C,q)=\sum \binom{m_j}{2},
\end{equation}
valid in any characteristic, where the sum is over all points $q_j$ equal to or infinitely near to $q$, see Formula (1) and section (1.3) in \cite{MW}. If $r(C,q)$ denotes the number of analytic branches of the singularity $(C,q)$, then one can define the corresponding Milnor number by the formula
\begin{equation}
\label{eq300}
\mu(C,q)=2 \delta(C,q)-r(C,q)+1,
\end{equation}
see section (1.3) in \cite{MW}. Assume now that $C$ is a reduced curve in the projective plane $\PP^2$ over an algebraically closed field. Let $q_1,\ldots,q_s$ be the points of $C$ which are singular and are actual points of $\PP^2$,
not infinitely near. Let $p_1,\ldots,p_r$ be the singular points of $C$, including all
infinitely near singular points. Thus $s\leq r$ and we may assume $q_1=p_1,\ldots,q_s=p_s$.
Then one can define, exactly as in section (1.3) in \cite{MW}, the total $\delta$-invariant of $C$ and the total Milnor number of $C$ using the formulas
\begin{equation}
\label{eq301}
\delta(C)=\sum_{j=1}^s \delta(C,q_j)=\sum_{i=1}^r\binom{m_i}{2},
\end{equation}
and
\begin{equation}
\label{eq302}
\mu(C)=\sum_{j=1}^s \mu(C,q_j).
\end{equation}
Let $\mu_i=\mu(C,q_i)$ and $r_i=r(C,q_i)$. Then one has by \eqref{eq300}
$$\mu_i=\sum_{j=1}^{r_i'}m_{ij} (m_{ij}-1)-r_i+1 $$
where $m_{ij}$ are the multiplicities of the infinitely near points to $q_i$.
If we replace this in the formula \eqref{eq302}, we get
\begin{equation}
\label{eq303}
\mu(C)=\sum_{i=1}^r m_i(m_i-1)-\sum_{j=1}^s(r_j-1)=\sum_{i=1}^r (m_i-1)^2-\sum_{j=1}^s(r_j-1)+\sum_{i=1}^r(m_i-1).
\end{equation}
Note that $r_i \leq m_i$ for all $i=1,...,s$, and hence we get
\begin{equation}
\label{eq304}
\mu(C)\geq \sum_{j=1}^r (m_i-1)^2.
\end{equation}
To prove the first claim of the Proposition, we set $k=2m_1-1$. Then one has
$$\sigma_k(C)=d^2-\sum_{i=1}^r(m_i^2-k) \geq d^2-\sum_{i=1}^r(m_i-1)^2 \geq d^2-\mu(C) \geq d^2-(d-1)^2=2d-1.$$
Indeed, one clearly has
$$m_i^2-k \leq (m_i-1)^2$$
for all $i=1,...,r$. On the other hand, the inequality $\mu(C) \leq (d-1)^2$ is well known in the complex case, see for instance the embedding of Milnor lattices (4.4.1) on p. 161 in \cite{DSTH}, or look at the papers \cite{Bru,D85}. To prove it in the case of characteristic $p>0$, we proceed as follows. Let $C'$ be a smooth curve in $\PP^2$ of degree $d$. Using the formula (3) in \cite{MW} we get
$$\chi(C)=2 \chi(\OO_C)+\mu(C)$$
and
$$ \chi(C')=2 \chi(\OO_{C'})=2(1-g)=2-(d-1)(d-2),$$
where $\chi(C)$ and $\chi(C')$ denote the corresponding $\ell$-adic Euler-Poincar\' e characteristic of $C$ and respectively of $C'$.
Since $C$ and $C'$ occur as fibers in a flat, proper family, it follows that
$$\chi(\OO_C)=\chi(\OO_{C'})=1-g$$
and hence
$$\mu(C)=\chi(C)-2 \chi(\OO_C)=\beta_0-\beta_1+\beta_2-2+(d-1)(d-2).$$
Here $\beta_j$ are the $\ell$-adic Betti numbers of $C$ and we know the following information, see section (1.3) in \cite{MW}.
\begin{enumerate}

\item $\beta_0=1$ since $C$ is connected.

\item $\beta_1 \geq 0$

\item $\beta_2=h$, the number of irreducible components of $C$. In particular $\beta_2 \leq d$.
\end{enumerate}
It follows that
$$\mu(C)\leq 1 +h -2 +(d-1)(d-2) \leq (d-1)^2$$
as we have claimed. This ends the proof of the first part of the Proposition.

For the second part, when $C$ is irreducible, we follow the same idea, just 
replacing the Milnor numbers with the $\delta$-invariants of the points $q_i$. 
So for $k=m_1$ we get the following inequalities:
$$\sigma_k(C)=d^2-\sum_{i=1}^r(m_i^2-k) \geq d^2-\sum_{i=1}^r(m_i-1)m_i \geq d^2-2\delta(C) \geq d^2-(d-1)(d-2)=3d-2.$$
Here we have used the obvious inequality
$m_i^2-m_1 \leq m_i(m_i-1)$ and
the genus formula, which holds in characteristic $p$ as well and implies $2\delta(C)\leq (d-1)(d-2))$.  
\end{proof}

\begin{proof}[Proof of Theorem \ref{thmNew}] 
For a rational plane curve $C$ of degree $d$ with $r$ singular points of multiplicities $m_1,\ldots,m_r$, 
we can bound $\sum_{i=1}^rm_i =2n_2+3n_3+4n_4+ \cdots+(d-1)n_{d-1}$ in terms of $r$ and $d$:
$$\sum_{i=1}^rm_i \leq \frac{r+\sqrt {r^2+4r(d-1)(d-2)}}{2}.$$
This is an equality exactly when all of the singular points have the same multiplicity.

To justify this inequality, let $r\mu(\mu-1)=(d-1)(d-2)$; then $\mu=\frac{r+\sqrt {r^2+4r(d-1)(d-2)}}{2}$.
But we have $(d-1)(d-2)=\sum_im_i(m_i-1)$ by the genus formula.
Thus $r\mu(\mu-1)=\sum_im_i(m_i-1)$, hence $r\mu\geq\sum_im_i$ 
with equality if and only if the $m_i$ are all equal (see for example \cite[Lemma 3.12]{BDHLPS}).

In particular, by Corollary \ref{cor0}, if 
$$3d-2>\frac{-3r+\sqrt {r^2+4r(d-1)(d-2)}}{2},$$
then Conjecture \ref{c2} holds.
This inequality is equivalent to 
$$(36-4r)d^2+d(r-1)48+r(8r-32)+16>0.$$
It is easy to check that this holds for all $d>0$ and $r\leq 9$, but fails for all $d\gg0$ if $r>9$.
Setting $r=10$, it is easy to see it holds if and only if $1\leq d\leq 109$.
\end{proof} 

\begin{proof}[Proof of Proposition \ref{prop1}] 
Assume  that $C$ is a rational plane curve, having only nodes and ordinary 
triple points as singularities. Let $d$ be the degree of $C$, so by the genus formula we have
$$n_2+3n_3=\frac{(d-1)(d-2)}{2}.$$
We have the following case-by-case discussion.
\vskip\baselineskip

\noindent(i) When $d=5$, there is no irreducible rational quintic with only 2 triple points.
Indeed, the line $L$ joining these two points would have intersection
$C \cdot L \geq 3+3=6$, and hence $L$ would be a component of $C$. This is not possible since $C$ is irreducible.
\vskip\baselineskip

\noindent(iii) When $d=6$, the case $n_2=0$ is not possible since $n_2+3n_3=10$.
(However, there is an irreducible sextic having $n_2=1$ and $n_3=3$; 
see Example 3.5 and Example 3.11 in \cite{BGI}. 
One can obtain such an example by applying a quadratic transformation
centered at the coordinate vertices to a nodal cubic 
for which each coordinate line meets the cubic
in three distinct points, such as
$F=x^3+4x^2y+4xy^2+y^3-6x^2z-2xyz-6y^2z-xz^2-yz^2+6z^3$,
% R=QQ[x,y,z]
% P=ideal(x-z,y-z);
% P1=ideal(y-z,x);
% P2=ideal(y+z,x);
% P3=ideal(x+z,y);
% P4=ideal(x-z,y);
% P5=ideal(x-6*z,y);
% P6=ideal(y-6*z,x);
% I=intersect(P^2,P1,P2,P3,P4,P5,P6);
% F=I_0;
%F=x^3+4*x^2*y+4*x*y^2+y^3-6*x^2*z-2*x*y*z-6*y^2*z-x*z^2-y*z^2+6*z^3
which has its node at $(1,1,1)$.
Applying the quadratic transformation
means substituting $1/x$ for $x$, $1/y$ for $y$ and $1/z$ for $z$,
and multiplying by $(xyz)^3$; this gives
the sextic 
$G=6x^3y^3-x^3y^2z-x^2y^3z-6x^3yz^2-2x^2y^2z^2-6xy^3z^2+x^3z^3+4x^2yz^3+4xy^2z^3+y^3z^3$.
%6*x^3*y^3-x^3*y^2*z-x^2*y^3*z-6*x^3*y*z^2-2*x^2*y^2*z^2-6*x*y^3*z^2+x^3*z^3+4*x^2*y*z^3+4*x*y^2*z^3+y^3*z^3
It still has a node at $(1,1,1)$ but it now has ordinary triple points at the coordinate vertices.)
\vskip\baselineskip

\noindent(iv) When $d=7$, there is no irreducible rational curve $C$ with only 5 triple points. Indeed, if such a curve exists, then choose a conic $Q$ passing through these 5 triple points. We have $C \cdot Q \geq 3 \cdot 5=15>14=\deg (C) \cdot \deg (Q)$. Hence $C$ and $Q$ would have a common irreducible component. This is not possible since $C$ is irreducible.
\end{proof}

\section{Examples} 

First we exhibit two standard examples referred to earlier.

\begin{ex}\label{xxx}
Here we give a sequence of rational curves $C_d$ of degree $d$
such that $H(C_d)$ approaches $-2$ asymptotically from above.
We define $C_d$ to be the image of $\PP^1$ under a general map of $\PP^1$ into $\PP^2$ given by 
three degree $d$ polynomials in $k[\PP^1]=k[x,y]$ over a field $k$.
Then $C_d$ has $\binom{d-1}{2}$ singular points of multiplicity 2. 
Thus there is no upper bound (in any characteristic) to the number of double points on
plane curves of genus 0. (This relates in a weak way to Bounded Negativity.
If $X\to\PP^2$ is the morphism given by
blowing up these points, then the proper transform $C_d'$ of $C_d$ on $X$ has $C'=d^2-2(d-1)(d-2)
=-d^2+6d-4$ and thus becomes arbitrarily negative as $d$ grows.)
We have
$$H(C_d)=\frac{d-2\binom{d-1}{2}}{\binom{d-1}{2}}=-2+\frac{d}{\binom{d-1}{2}}>-2,$$
thus $H(C_d)$ achieves $-2$ only in the limit.
\end{ex}

\begin{ex}\label{yyy}
Here we give an easy construction of reduced curves $C$ to show that $H(C)$ can be arbitrarily negative 
in any positive characteristic. 
Take $C$ to be the union of all lines defined over a finite subfield $k$ of the ground field $K$
when $K$ has positive characteristic. If $q$ is the order of $k$, 
then there are $q^2+q+1$ lines, hence $C$ has degree $q^2+q+1$. The singular points of
$C$ are precisely the $k$-points of $\PP^2_k$ and each has multiplicity $q+1$, so
$H(C)=(q^2+q+1-(q^2+q+1)(q+1))/(q^2+q+1)=-q$.
\end{ex}

\begin{ex}\label{zzz}
A reduced, irreducible rational curve $C$ of degree $d$ with only triple points would,
by Equation \ref{genusEquation}, have 
$r=(1/3)\binom{d-1}{2}$, hence 
$H(C)=\frac{3d-2-\binom{d-1}{2}}{(1/3)\binom{d-1}{2}}=-3+\frac{3(3d-2)}{\binom{d-1}{2}}$,
and this approaches $-3$ asymptotically from above as $d$ increases.
\end{ex}

\subsection{The multiplicity sequences of the form $(d,3_r)$ and genus bound 0 currently known
for complex plane curves}
Here we exhibit complex curves all of whose singular points are triple points.
We do this for every multiplicity sequence for which we are aware of such a curve.
It happens that in each case we can give a curve with only ordinary singularities.
The examples we give here are not new; we include explicit constructions for the convenience of the reader.
%$(4,3)$, $(5,3,3,3)$, $(6,3,3,3,3)$, $(8,3_7)$ and $(8,3_8)$.

We begin with reduced, irreducible examples.
Figure \ref{rose} shows a reduced, irreducible complex curve having multiplicity sequence $(4,3)$.

\begin{figure}[h]
\begin{tikzpicture}[line cap=round,line join=round,x=1.5cm,y=1.5cm]
\draw[thick,variable=\t,domain=-100:100,samples=103] plot ({cos(\t)*sin(3*\t)},{sin(\t)*sin(3*\t)});
\end{tikzpicture}
\caption{The rose $r=\sin(3\theta)$, or $(x^2+y^2)^2+y^3-3yx^2=0$.}
\label{rose}
\end{figure}
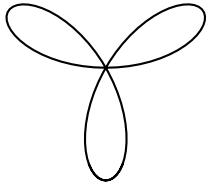

Figure \ref{irredOctic} shows a reduced, irreducible complex curve having multiplicity sequence $(8,3_7)$.
It is an example of a homaloidal curve, meaning that it can be brought to a straight line by a Cremona transformation.
Using this fact one can give an explicit equation for it. The curve $C$ we show here
was chosen not because its equation was attractive, but rather because its graph makes our claims obvious.
It is clear from the graph that the curve is irreducible of degree 8. Its graph shows it has at least 7 
singular points of multiplicity at least 3 each. By the genus formula, this already makes the curve rational,
and shows there are no further singularities and that the multiplicities are at most 3 and now we see the
the points are ordinary singular points.
Taking $z=1$, the affine coordinates $(x,y)$ of the singular points are: 
$(-3,2)$,
$(-2.4,3.6)$,
$(-1.2,-0.2)$,
$(0.5,3)$,
$(1.4,4.2)$,
$(2,-0.1)$ and
$(2.5,2.25)$. The curve also contains the smooth points
$(-2.8,0.6)$ and
$(1.8,1.9)$.
The number of conditions imposed by 7 triple points and two smooth points on the $\binom{8+2}{2}=45$ dimensional 
space of octics is at most $7\binom{3+1}{2}+2=44$. Thus there is at least one such octic; Macaulay2 \cite{M2}
confirms there is exactly one and by upper semicontinuity, for a general set of 9 points,
there will be exactly one octic curve for which the first 7 points are triple points and the last two points are smooth.

The general octic $C$ with triple points at 7 general points can be brought to a line by
applying quadratic transforms as follows.
Apply a quadratic transform centered at three of the triple points.
The multiplicity sequence of the image $C_1$ of $C$ is $(7,2_3, 3_4)$.
Applying a quadratic transform centered at three of the triple points of $C_1$
gives a curve $C_2$ with multiplicity sequence $(5,3,2_3)$
(the three triple points became smooth).
Applying a quadratic transform centered at the three double points
of $C_2$ gives a curve $C_3$ with multiplicity sequence $(4,3)$
(the three double points became smooth). 
Applying a quadratic transform centered at the triple point
and at two smooth points of $C_3$ gives a curve $C_4$
with multiplicity sequence $(3,2)$.
Applying a quadratic transform centered at the double point
and at two smooth points of $C_4$ gives a curve $C_5$
which is a smooth conic.
Applying a quadratic transform centered at the any three points
of $C_5$ gives a curve $C_6$
which is a line. Running the sequence of quadratic transforms in reverse
takes a line to the octic.

\begin{figure}[h]
\includegraphics[scale=.25]{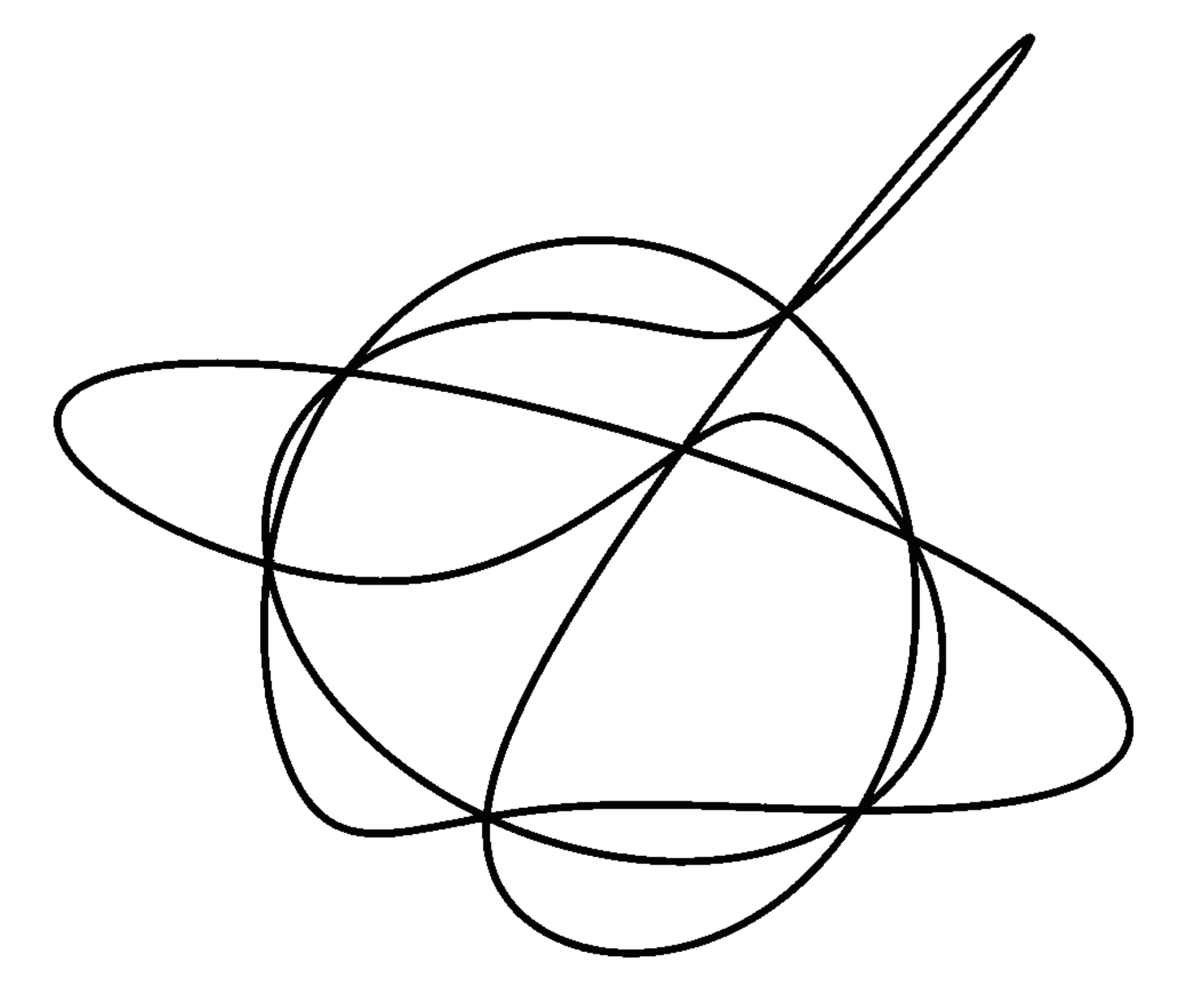}
\caption{Irreducible complex rational octic with 7 triple points.}
\label{irredOctic}
\end{figure}

Now we consider reduced, reducible examples.
First there is the case of 3 concurrent lines, so three members of a pencil of lines. 
The multiplicity sequence for this case is $(3,3)$.

Figure \ref{Conic3Lines} shows a reduced curve having multiplicity sequence $(5,3,3,3)$.

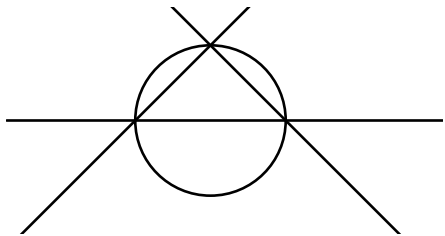
\begin{figure}[h]
\begin{tikzpicture}[line cap=round,line join=round,x=.5cm,y=.5cm]
\clip(-5.42,-3) rectangle (6.18,3);
\draw [line width=1.pt] (0.,0.) circle (1.cm);
\draw [line width=1.pt,domain=-5.42:6.18] plot(\x,{(--4.--2.*\x)/2.});
\draw [line width=1.pt,domain=-5.42:6.18] plot(\x,{(--4.-2.*\x)/2.});
\draw [line width=1.pt,domain=-5.42:6.18] plot(\x,{(-0.-0.*\x)/-4.});
\end{tikzpicture}
\caption{A conic and 3 lines with multiplicity sequence $(5,3,3,3)$.}
\label{Conic3Lines}
\end{figure}

Figure \ref{3conics} shows a reduced curve having multiplicity sequence $(6,3,3,3,3)$.
It consists of the union of 3 irreducible conics in a pencil of conics.

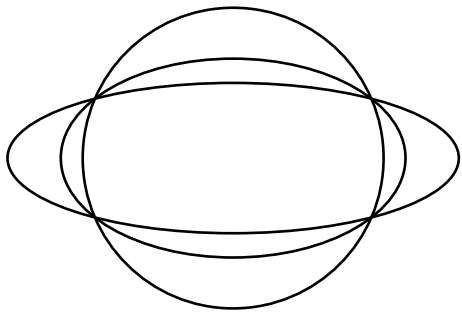
\begin{figure}[h]
\begin{tikzpicture}[line cap=round,line join=round,x=0.5cm,y=0.5cm]
\clip(-6.2,-4.46) rectangle (6.18,4.92);
\draw [line width=1.pt] (0.,0.) circle (2.cm);
\draw [rotate around={0.:(0.,0.)},line width=1.pt] (0.,0.) ellipse (3.cm and 1.cm);
\draw [rotate around={0.:(0.,0.)},line width=1.pt] (0.,0.) ellipse (2.29128784747792cm and 1.3228756555322954cm);
\end{tikzpicture}
\caption{Three conics in a pencil of conics with multiplicity sequence $(6,3,3,3,3)$: $y^2 + x^2 - 4=0$, $x^2 - 9 + 9y^2=0$
and $x^2 - 9 + 9y^2 + 3 (y^2 + x^2 - 4) = 0$.}
\label{3conics}
\end{figure}

Figure \ref{4conics} shows a reduced curve having multiplicity sequence $(8,3_8)$,
two of which are bolded merely for visual clarity.
It consists of the union of 4 conics, each conic containing 6 of the 8 triple points.
An example of such conics is
$9x^2+16y^2-300^2z^2=0$, $16x^2+9y^2-300^2z^2=0$, 
$16x^2-7xy+16y^2-300^2z^2=0$ and $16x^2+7xy+16y^2-300^2z^2=0$.
These conics are components of four cubics in the pencil of cubics whose 9 base points
are the 8 singular points of the union of the four conics
(namely $(\pm75:0:1)$, $(0,\pm75:1)$, $(\pm60:\pm60:1)$) and the point $(0:0:1)$.
The four cubics are
$A=x(16x^2+9y^2-300^2z^2)$, $B=y(9x^2+16y^2-300^2z^2)$, $C=(x+y)(16x^2-7xy+16y^2-300^2z^2)$
and $D=(x-y)(16x^2+7xy+16y^2-300^2z^2)$.
Note that $A+B=C$ and $A-B=D$.

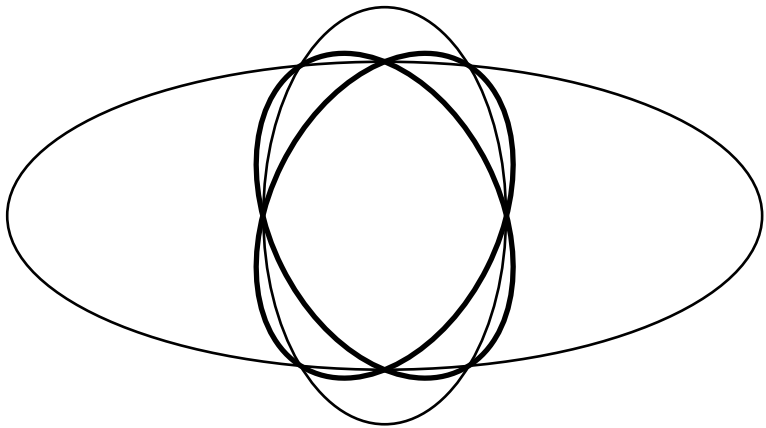
\begin{figure}[h]
\begin{tikzpicture}[line cap=round,line join=round,x=1.0cm,y=1.0cm]
\clip(-5.236226018427141,-3.576561003249869) rectangle (5.478112567456207,3.4192718381209963);
\draw [rotate around={0.:(0.,0.)},line width=1.pt] (0.,0.) ellipse (5.018148879997303cm and 2.048650701552048cm);
\draw [rotate around={116.56505117707799:(0.,0.)},line width=2.pt] (0.,0.) ellipse (2.29128784747792cm and 1.5275252316519468cm);
\draw [rotate around={-116.56505117707799:(0.,0.)},line width=2.pt] (0.,0.) ellipse (2.29128784747792cm and 1.5275252316519468cm);
\draw [rotate around={90.:(0.,0.)},line width=1.pt] (0.,0.) ellipse (2.7745119690977487cm and 1.613106529667628cm);
\end{tikzpicture}
\caption{Four conics giving an octic with multiplicity sequence $(8,3_8)$.}
\label{4conics}
\end{figure}

We can obtain the multiplicity sequence $(9,3_{10})$ using a curve in a Halphen pencil of order 3.
A Halphen pencil of order $m>1$ is a pencil of curves of degree $3m$ with 9 base points
where the general member is irreducible. There is a unique cubic through the 9 base points;
one of the members of the Halphen pencil is thus this cubic taken with multiplicity $m$.
For our example we want $m=3$, we want none of the 9 base points to be infinitely near another, 
and we want the pencil to have a reducible member $C$ consisting of three rational curves
which together have ordinary singularities at the base points and which
meet in one more ordinary singular point away from the 9 base points. (Thus $C$ gives rise to a 
curve $C'$ in the elliptic fibration 
corresponding to the Halphen pencil, where $C'$ has the topological type of three concurrent lines.
In Kodaira's classification of singular fibers
of elliptic fibrations, $C'$ has type IV.) 

It is easy to construct a pencil of cubics having a type IV member; see Figure \ref{CubicPencil}.
It is also easy to construct a Halphen pencil of order 2 with a type IV member:
basically take three conics with a given point in common (see Figure \ref{HalphenPencilOrd2};
also see \cite{Z} for constructions of various Halphen pencils of order 2).
Constructing examples of order $m>2$ is much harder. One of the problems in our case is that
it is not easy to ensure that one gets a curve of type IV instead of type $I_3$
(the latter has the topological type of three lines in the plane which are not concurrent).
We now sketch the construction of a Halphen pencil with a nonic having 10 ordinary triple points.

We will pick 9 points on the nodal cubic $A: 27xyz-(x+y+z)^3=0$.
This cubic $A$ has a node at $(1,1,1)$ and flexes at $F_0=(-1,1,0), F_1=(-1,0,1)$ and $F_2=(0,-1,1)$.
The curve $A$ will occur with multiplicity 3 in the Halphen pencil we will construct.
Using a nodal cubic has the advantage of having a parameterization and hence a lot of rational points.
Using this particular nodal cubic has the advantage that its flexes and node are rational
(but the branches at the node, unfortunately, are not real).

Our pencil will also have a reducible nonic curve $C$ of type IV. Two of the components of
$C$ will be irreducible quartics $Q_1$ and $Q_2$ with three nodes each. The third component of $C$ will be a line, $L$.
The 9 base points $p_1,\ldots, p_9$ of the pencil will be chosen so that
$Q_1$ has nodes at $p_1,p_2,p_3$ and is smooth at $p_4,\ldots,p_9$, while 
$Q_2$ has nodes at $p_4,p_5,p_6$ and is smooth at the other 6 base points.
In addition $L$ will contain $p_7,p_8$ and $p_9$, and all three curves, $Q_1, Q_2$ and $L$,
will meet pairwise transversely at a tenth point $p_{10}$.

We can parameterize $A$ as $x=-(t+1)^3$, $y=t^3$ and $z=1$. 
If we use $(-1,1,0)$ as the identity in the group law on the smooth points of $A$,
the inverse of a point $(a,b,c)$ on $A$ is $(b,a,c)$.
Given smooth points $(-(t+1)^3,t^3,1)$ and $(-(s+1)^3,s^3,1)$, the line through these points
meets $A$ in a third point, and this third point is  $((st-1)^3,-(st+s+t)^3,(s+t+1)^3)$.
Thus the group law is 
$$(-(t+1)^3,t^3,1) + (-(s+1)^3,s^3,1)=(-(st+s+t)^3,(st-1)^3,(s+t+1)^3).$$

In the group law on $A$, the intersection of $Q_1$ and $A$ can be written as
$2(p_1+p_2+p_3)+p_4+\cdots+p_9=0$.  
Then the intersection of $Q_2$ and $A$ is
$p_1+p_2+p_3+2(p_4+p_5+p_6)+p_7+p_8+p_9=0$.  
And the intersection of $L$ and $A$ is $p_7+p_8+p_9=0$.
Subtracting the second from the first gives $p_1+p_2+p_3-(p_4+p_5+p_6)=0$,
Subtracting the third from the first and adding the fourth gives $3(p_1+p_2+p_3)=0$,
and hence also $3(p_4+p_5+p_6)=0$.
We do not want $p_1+p_2+p_3=0$ or $p_4+p_5+p_6=0$ since this would mean the points
are collinear (which would mean that $Q_1$ and $Q_2$ would not be irreducible and thus we would not get a type IV fiber).
Thus we want $p_1+p_2+p_3$ to have 3-torsion, so $p_1+p_2+p_3=F_1$, for example, and likewise
$p_4+p_5+p_6=F_1$. Thus we want $p_3=F_1-p_1-p_2$ and $p_6=F_1-p_4-p_5$.
(We could use $F_2$ in place of $F_1$, but we cannot take $p_1+p_2+p_3=F_1$ and 
$p_4+p_5+p_6=F_2$, since then $p_1+\cdots+p_6=F_1+F_2=0$, which means the
points $p_1,\ldots,p_6$ lie on a conic, which would mean the conic is a component of both $Q_1$ and $Q_2$.)

So we get to pick $p_1,p_2,p_4,p_5,p_7,p_8\in A$; 
then $p_1,p_2$ determine $p_3$, while
$p_4,p_5$ determine $p_6$, and
$p_7,p_8$ determine $p_9$ (since $p_7,p_8,p_9$ are where $L$ and $A$ intersect).
Now $Q_1$ and $Q_2$ meet $L$ in 4 points each, including $p_7,p_8,p_9$.
One must be careful to pick $p_8$ so that $Q_1$ and $Q_2$ both meet $L$ in the same
fourth point, which we will designate as $p_{10}$.

Here are the choices we made:
$p_1=(-8,1,1)$,
$p_2=(-216,125,1)$,
$p_4=(-64,27,1)$,
$p_5=(-125,64,1)$,
$p_7=(-1,1,0)$ and
$p_8=(-(t+1)^3,t^3,1)$.
This forces
$p_3=(-11^3,7^3,4^3)$,
$p_6=(-19^3,8^3,11^3)$ and
$p_9=(t^3,-(t+1)^3,1)$.
By computation, the condition on $t$ 
to guarantee that $Q_1$ and $Q_2$ both meet $L$
at the same fourth point $p_{10}$ is
$117008(t+\frac{1}{2})^6+4295799(t+\frac{1}{2})^4-\frac{3637629}{2}(t+\frac{1}{2})^2+\frac{4822335}{16}=0$.
% 117008*t^6+351024*t^5+4734579*t^4+8884118*t^3+4734579*t^2+351024*t+117008
% 117008*(u-1/2)^6+351024*(u-1/2)^5+4734579*(u-1/2)^4+8884118*(u-1/2)^3+4734579*(u-1/2)^2+351024*(u-1/2)+117008
% 117008*u^6+4295799*u^4-(3637629/2)*u^2+4822335/16
% 117008*(t+(1/2))^6+4295799*(t+(1/2))^4-(3637629/2)*(t+(1/2))^2+4822335/16
(This is a cubic in a square, so in principle we can find $t$ exactly. In fact, two of the roots
give a purely imaginary conjugate pair and the other 4 comprise two conjugate pairs which are not
real and not purely imaginary.)

One can now check by brute force computation that the union of $Q_1$, $Q_2$ and $L$
is a curve of degree 9 with 10 ordinary triple points (namely, $p_1,\ldots,p_{10}$) and no other singular points.
The fact that $Q_1$ (and likewise for $Q_2$)
is a quartic whose only singular points are three noncollinear nodes
implies that it is irreducible (consider cases to verify this; for example a cubic and a line can have 3 nodes,
but the nodes are collinear).
The points $p_8$ and $p_9$ are not real, and probably $p_{10}$ is not also.
Moreover, $L$ is not real, and probably neither are $Q_1$ and $Q_2$.
(It would be interesting to know if there is a real example of a curve of degree 9 with 10
triple points.)

\begin{figure}[h]
\definecolor{xdxdff}{rgb}{0.49019607843137253,0.49019607843137253,1.}
\begin{tikzpicture}[line cap=round,line join=round,x=1.0cm,y=1.0cm]
\clip(-4.44,-3) rectangle (7.16,4.44);
\draw [dashed,line width=1,  domain=-0.03:4, samples=100] plot ({\x}, {sqrt(2*\x*(\x-2)^2+0.25)} );
\draw [dashed,line width=1,  domain=-0.03:4, samples=100] plot ({\x}, {-sqrt(2*\x*(\x-2)^2+0.25)} );
\draw [dashed,line width=1,  domain=-.01:.01, samples=100] plot ({-0.03}, {\x} );
\draw [line width=1.pt,domain=-4.44:7.16] plot(\x,{(--0.6666666666666666--0.6666666666666666*\x)/1.});
\draw [line width=1.pt,domain=-4.44:7.16] plot(\x,{(--0.3333333333333333--0.3333333333333333*\x)/1.});
\draw [line width=1.pt,domain=-4.44:7.16] plot(\x,{(-0.3333333333333333-0.3333333333333333*\x)/1.});
\draw [fill=xdxdff] (0.045,0.7) circle (2.5pt);
\draw [fill=xdxdff] (1.1323076923076922,1.4215384615384614) circle (2.5pt);
\draw [fill=xdxdff] (3.0984615384615384,2.732307692307692) circle (2.5pt);
\draw [fill=xdxdff] (-0.01,0.33) circle (2.5pt);
\draw [fill=xdxdff] (1.604,0.868) circle (2.5pt);
\draw [fill=xdxdff] (2.462,1.154) circle (2.5pt);
\draw [fill=xdxdff] (-0.01,-0.33) circle (2.5pt);
\draw [fill=xdxdff] (1.604,-0.868) circle (2.5pt);
\draw [fill=xdxdff] (2.462,-1.154) circle (2.5pt);
\end{tikzpicture}
\caption{The nine base points of a cubic pencil defined by a Kodaira type IV member (i.e., three concurrent lines)
and a smooth cubic (dashed), with a third smooth cubic in the pencil (dotted), giving a
curve with multiplicity sequence $(9,3_{10})$ and genus bound 1.
Including another smooth cubic curve from the pencil gives an overall curve
with multiplicity sequence $(9,3_{10})$ and genus bound 1.}
\label{CubicPencil}
\end{figure}
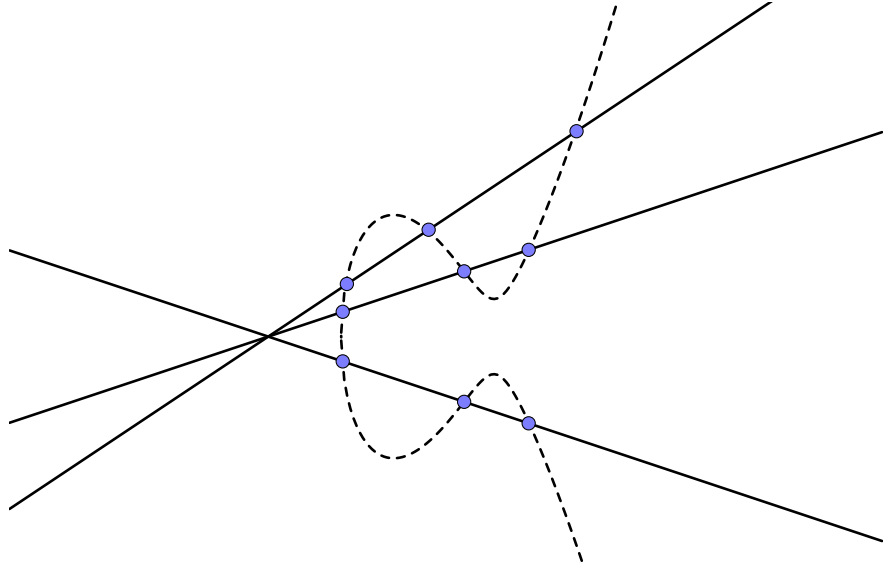

\begin{figure}[h]
\includegraphics[scale=.5]{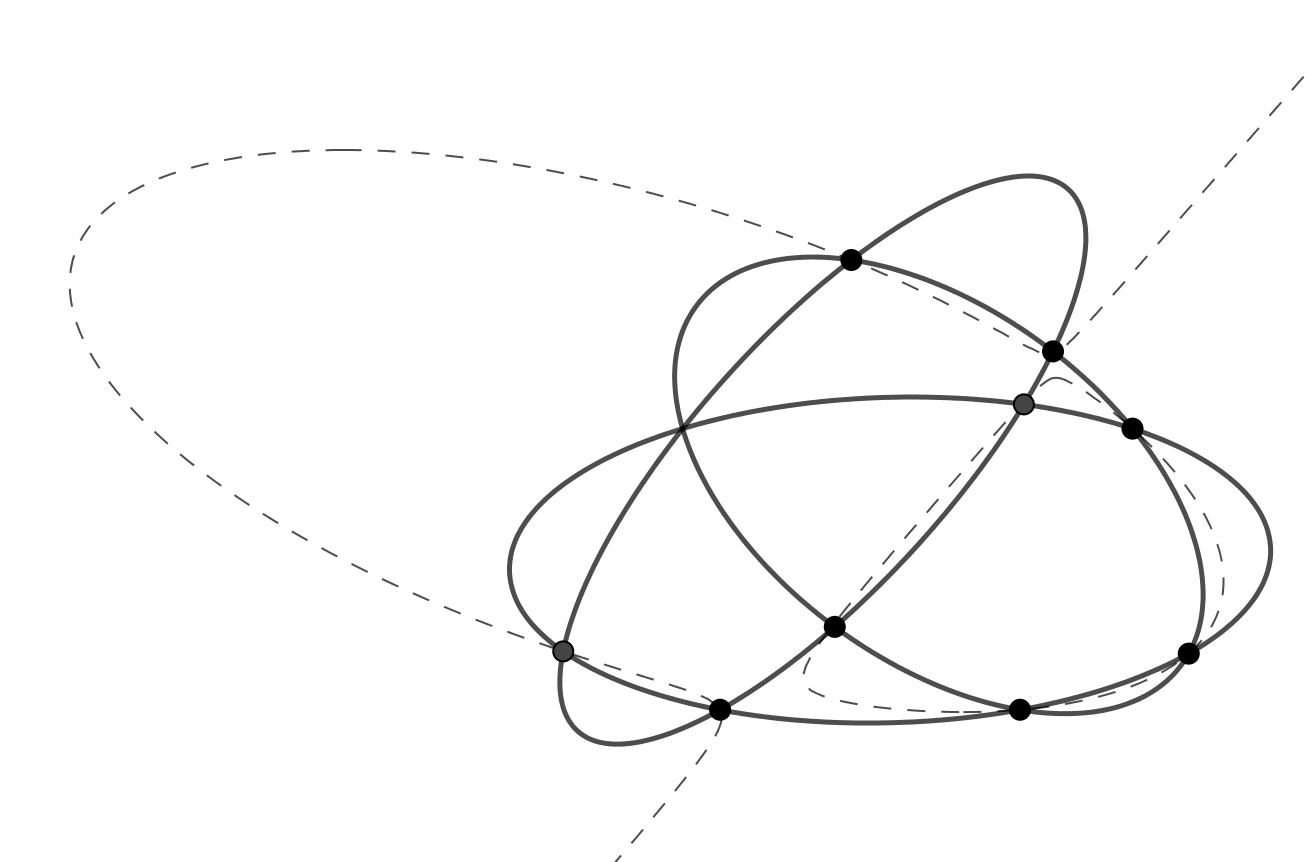}
\caption{The nine base points of a pencil of sextics (i.e., a Halphen pencil of order 2) with a 
Kodaira type IV member (viz., the three concurrent conics) and a smooth cubic
of multiplicity 2 (dashed). The sextic and the cubic together give a curve
with multiplicity sequence $(9,3_{10})$ and genus bound 1.}
\label{HalphenPencilOrd2}
\end{figure}

Finally, we have the union of 9 lines, consisting of the three singular members of the pencil of cubics given
by $x^3-y^3$ and $y^3-z^3$, the other singular cubic in the pencil being the sum, $x^3-z^3$.
This curve cannot be defined over the reals, so we do not provide a sketch,
but the singular points consist of the 9 base points of the pencil plus the 3 coordinate vertices of $\PP^2$.
This curve has multiplicity sequence $(9,3_{12})$. 

\begin{rk}
All of the examples of complex curves with only triple point singularities that we know of fall into one of two categories.
The irreducible curves are homaloidal, and the rest are related to pencils of curves, either a pencil of lines, conics, cubics
or an elliptic pencil. For example, the multiplicity sequence $(3,3)$ is represented by three members of a pencil of lines.
For $(5,3_3)$ we have the union of a conic and three lines.
Adding a line to the conic gives two cubics which define a pencil; the conic is a component of one member of the pencil
and the lines comprise another member. The case $(6,3_4)$ comes from three members of a pencil of conics,
while $(8,3_8)$ comes from taking one component from each of four reducible members of a pencil of cubics.
Next, $(9,3_{10})$ comes from a reducible member of an elliptic pencil and $(9,3_{12})$ is given by 
three reducible members of a pencil of cubics.

A similar situation holds for the examples $(7,3_7)$, which arises in characteristic 2, and the
version of $(9,3_{12})$ which arises in characteristic 3,
except the pencils in question are quasi-elliptic.

The sequence $(7,3_7)$ consists of the lines of the Fano plane (here the characteristic is 2). 
If $p$ is a general point of the plane, there is a unique cuspidal cubic $C$ with its cusp at $p$ 
where $C$ also contains the 7 points of the Fano plane. The 7 points together with $p$ are the base points 
of a unique cubic pencil (the ninth base point of the pencil is infinitely near to $p$). 
The pencil has 7 reducible members, each one consisting of one of the 7 lines 
together with a conic. After blowing up the 9 base points
we get a quasi-elliptic fibration, whose general fiber is a cuspidal cubic and which has 8 reducible fibers
corresponding to the 7 lines with their respective conics and the fiber coming from $C$.

Finally, consider the case $(9,3_{12})$ which arises in characteristic 3.
The 9 points of the finite projective plane over ${\mathbb Z}/3{\mathbb Z}$ away from infinity (i.e., the points
$(a,b,1)$ where $a$ and $b$ are each either $-1$, 0 or 1) determines a quasi-elliptic pencil of cubics with
4 reducible members, each of which consists of three lines which meet at infinity, with one line of each triple
containing the point $(0,0,1)$. Then $(9,3_{12})$ comes from taking these reducible members,
excluding one component from each one (viz., the component that contains $(0,0,1)$) 
and adding in the line at infinity (i.e., $z=0$). In fact, the construction here is very similar to that
of $(8,3_8)$ in characteristic 0, except instead of four irreducible conics, the four conics are reducible
and their singular points are collinear. Adding the line through the four singular points 
increases the degree from 8 to 9 and adds 4 more triple points, thus $(8,3_8)$
becomes $(9,3_{12})$.
\end{rk}

\begin{rk}
If we ignore infinitely near points, there are additional examples of curves
with only triple points. For example, 
let $C$ be the union of three smooth members of the pencil of conics defined by $yz-x^2$ and $x^2$.
If we ignore infinitely near points, this has one singular point, and this point has
multiplicity 3. But in fact, after taking into account infinitely near points, 
its multiplicity sequence is $(6,3_4)$, the same as in Figure \ref{3conics}.

Similarly, the curve $F=x^5+4y^5+4xy^3z+x^2yz^2=0$ is reduced and irreducible,
and has only one singular point, a triple point, if we ignore infinitely near points, but its
multiplicity sequence is $(5,3,2,2,2)$ since it has infinitely near double points. 
It is reduced since one can resolve the triple point
by a sequence of blow ups. It is irreducible since the resolution by blow ups of 
its sole singular point shows the singular point has only two branches, thus if
$C$ were reducible it would have 2 components,
whose tangents would be $x$ and $y$. Neither $x$ nor $y$ divides $F$, so
$F$ must be a conic and a cubic. Since there are only two branches,
the cubic must have a cusp, hence giving the tangent cone $x^2$,
and the conic must have tangent cone $y$.
But then the intersection multiplicity of the components at the triple point
is 2, which means the $C$ would have another singular point elsewhere.
Thus $C$ must be irreducible.
\end{rk}

\begin{rk}
We comment here on the most negative currently known values of $H$-constants.
For irreducible $C$, the infimum for $H(C)$ is at most $-2$ \cite{BDHLPS}, 
but no cases with $H(C)\leq-2$ are known.
For reducible $C$, the infimum for $H(C)$ is at most $-4$ \cite{R},  
but no cases with $H(C)\leq-4$ are known.
For reducible $C$ with only ordinary singularities, 
the infimum for $H(C)$ is at most $-25/7$ \cite{PR}, 
but no cases with $H(C)\leq-25/7$ are known.

For reducible $C$ whose components are all lines, 
the most negative example known has $H(C)=-225/67\approx -3.358$ 
and there is a theoretical bound $H(C)>-4$ \cite{BDHLPS}.
Indeed, by Proposition \ref{genusFormula} we have 
$$H(C)= \frac{d-\sum_{i=1}^rm_i}{r}.$$
In order for this to be $-4$ or less, the excess beyond 4 of the sum of the terms $m_i$ 
with $m_i>4$ must be at least as much the sum of $d$ and the terms $m_i$ equal to 2 or 3.
I.e., $H(C)\leq -4$ if and only if 
$2n_2+n_3+d \leq \sum_{k>4}(k-4)n_k$.
It is easy to see $2n_2+n_3+d > \sum_{k>4}(k-4)n_k$ 
if the lines are concurrent or if all of but one of the lines 
are concurrent. To see that $2n_2+n_3+d > \sum_{k>4}(k-4)n_k$
otherwise, the proof uses a result of Hirzebruch:
$$n_2+(3/4)n_3\geq d+\sum_{k>3}(k-4)n_k.$$
We thus have $2n_2+n_3+d > n_2+(3/4)n_3-d\geq \sum_{k>4}(k-4)n_k$.

In case $C$ has components which have degree at most 2
and all singularities are ordinary, under some mild additional assumptions 
we have $H(C)\geq -4.5$ \cite{PT, PS}, but the most negative value known
is $-225/68\approx -3.309$ \cite{PRS}.
When $C$ is reducible, the components of $C$ are smooth and the singularities are ordinary,
from \cite{PRS} we have $H(C) \geq 4.5d-2.5d^2-4$. 
\end{rk}

\begin{rk}
Only four examples of complex line arrangements are known which have no
points of multiplicity exactly 2 \cite{BDHLPS}. These are: 3 or more concurrent lines;
the curves given by $(x^n-y^n)(x^n-z^n)(y^n-z^n)=0$ for $n>2$;
an arrangement due to F. Klein with 21 lines and $n_3=28$ and $n_4=21$;
and an arrangement due to A. Wiman with 45 lines and $n_3=120$, $n_4=45$ and $n_5=36$.

In contrast, there are lots of singular curves with $n_2=0$, even irreducible curves. 
For example, for each $d>2$ there is a 
reduced rational curve $C$ of degree $d$ with an ordinary singularity
of multiplicity $d-1$. 

For 9 sufficiently general points, there are infinitely many multiplicity sequences
for reduced irreducible rational curves with $n_2=0$, for example,
$(3m(m+1),m^2+2m,(m^2+m)_7,m^2-1)$ for every $m>1$.
One can show that these are all homaloidal, and hence irreducible and rational.
If one blows up the 9 base points of a general pencil of cubics, 
these curves still arise, and the proper transform of each of them
gives a section of the elliptic fibration given by the pencil of cubics.
They are thus points on the generic fiber, and translation under
the group law on the generic fiber gives many additional examples of rational curves
whose image in the plane have no double points,
but the multiplicities of their singular points are not bounded.
\end{rk}

%\newpage


\begin{thebibliography}{00}

\bibitem{BDHLPS} Th. Bauer, S. Di Rocco, B. Harbourne, J. Huizenga, A. Lundman, P. Pokora and T. Szemberg,
Bounded Negativity and Arrangements of Lines, International Math. Res. Notices (2015) 9456--9471.

\bibitem{BHKMS} Th. Bauer, B. Harbourne, A.L. Knutsen, A. K\"uronya, S. M\"uller-Stach and T. Szemberg.
Negative curves on algebraic surfaces, Duke Math. J., 162(2013), 1877--1894.

\bibitem{BHRS}Th. Bauer, B. Harbourne, J. Ro\'e 
and T. Szemberg. The Halphen cubics of order two,
Collectanea Mathematica  68(2017), 339--357.

\bibitem{BGI} A. Bernardi, A. Gimigliano, Monica Id\`a,
Singularities of plane rational curves via projections,
Journal of Symbolic Computation, 86(2018),189--214.

\bibitem{Bru} J.W. Bruce, An upper bound for the number of singularities on a projective hypersurface, Bull.
London Math. Soc, 13 (1981), 47--50.

\bibitem{CV}
R. Cheng, R. van Dobben de Bruyn,
Unbounded negativity on rational surfaces in positive characteristic,
preprint, 2021 (arXiv:2103.02172).

\bibitem{D85}  A. Dimca,   On isolated singularities of complete intersections, J. London Math. 
Soc. 31 (1985), 401--406.

\bibitem{DSTH}  A. Dimca,   {\em Singularities and Topology of Hypersurfaces} , Universitext, Springer Verlag, New York, 1992.

% \bibitem{DS14} A. Dimca, E. Sernesi,  Syzygies and logarithmic vector fields along plane curves, Journal de l'\'Ecole polytechnique-Math\'ematiques 1(2014), 247-267.

\bibitem{Dumi} O. Dumitrescu, Plane Curves with Prescribed Triple Points: A Toric Approach, Communications in Algebra, 41(2013),1626--1635.

% \bibitem{duPCTC} A.A. du Plessis,  C.T.C. Wall, Application of the theory of the discriminant to highly singular plane curves, Math. Proc. Cambridge Phil. Soc.,  126(1999), 259-266. 

% \bibitem{F} Fulton, W.: Introduction to Toric Varieties. Annals of Math. Studies, \textbf{131}, Princeton University Press, Princeton  (1993)

\bibitem{GLS} G.-M. Greuel, C. Lossen, E. Shustin,  Singular algebraic curves. Springer,
Switzerland (2018)

\bibitem{M2} 
D.R. Grayson and M.E. Stillman,
Macaulay2: a software system for research in algebraic geometry,
Available at \url{http://www.math.uiuc.edu/Macaulay2/}.
          
\bibitem{GS} 
G.-M. Greuel,  E. Shustin, 
Plane algebraic curves with prescribed singularities, arXiv:2008.02640.

\bibitem{Hao}  F. Hao, 
Weak bounded negativity conjecture. Proc. Amer. Math. Soc. 147.8 (2019) 3233--3238.

\bibitem{H}  B. Harbourne, 
Algebraic Geometry, Commutative Algebra and Combinatorics: interactions and open problems, preprint (2020).

\bibitem{Hr}  R. Hartshorne, 
Algebraic Geometry, Springer Verlag, Graduate Texts in Mathematics 52 (1977) XVI + 496 pp.

\bibitem{MW}  A. Melle-Hern\' andez, C.T.C. Wall,  Pencils of curves on smooth surfaces. Proc. Lond. Math. Soc., III. Ser. 83(2) (2001), 257--278 

\bibitem{Oka} M. Oka, On Fermat curves and maximal nodal curves, Michigan Math. J. 53 (2005), 459--477.

\bibitem{P} P. Pokora, Harbourne constants and arrangements of lines on smooth hypersurfaces in $\PP^3(\C)$,
Taiwan J. Math. 20(2016), 25--31.

\bibitem{PR} P. Pokora and J. Ro\'e,
Harbourne constants, pull-back clusters and ramified morphisms,
Results Math 74, 109 (2019). 

\bibitem{PRS} P. Pokora, X. Roulleau, T. Szemberg
Bounded negativity, Harbourne constants and transversal arrangements of curves,
Ann. Inst. Fourier 67:6 (2017) 2719-2735.

\bibitem{PS} P. Pokora and T. Szemberg,
Conic-line arrangements in the complex projective plane,
preprint (arXiv:2002.01760).

\bibitem{PT} 
P. Pokora, H. Tutaj-Gasi\'nska,
Harbourne constants and conic configurations on the projective plane,
Mathematische Nachrichten 289:7 (2016) 888--894.

\bibitem{R} X. Roulleau. Bounded negativity, Miyaoka-Sakai inequality and elliptic curve configurations,
Int. Math. Res. Not. 2017 (2017), no. 8, p. 2480--2496.

\bibitem{RU} X. Roulleau, G. Urzu\' a, Chern slopes of simply connected complex surfaces of general
type, Ann. Math. 182 (2015), 287--306.

% \bibitem{Steen} J. H. M. Steenbrink, Intersection form for quasi-homogeneous singularities, Compositio. Math., 34(1977), 211--223.

\bibitem{Szp1} J. Szpond. On linear Harbourne constants, British Journal of Mathematics \& Computer Science,
8(2015), 286--297.

\bibitem{CTC} C.T.C. Wall, Singular Points of Plane Curves, London Math. Soc. Student Texts 63, CUP 2004.

\bibitem{Z} 
A. Zanardini, Explicit constructions of Halphen pencils, preprint (arXiv:2008.08128).


\end{thebibliography}
\end{document}